\documentclass[10pt]{article}
\font\smallit=cmti10
\font\smalltt=cmtt10

\usepackage{amssymb,latexsym,amsmath,epsfig,amsthm} 

\makeatletter

\renewcommand\section{\@startsection {section}{1}{\z@}
{-30pt \@plus -1ex \@minus -.2ex}
{2.3ex \@plus.2ex}
{\normalfont\normalsize\bfseries\boldmath}}

\renewcommand\subsection{\@startsection{subsection}{2}{\z@}
{-3.25ex\@plus -1ex \@minus -.2ex}
{1.5ex \@plus .2ex}
{\normalfont\normalsize\bfseries\boldmath}}

\renewcommand{\@seccntformat}[1]{\csname the#1\endcsname. }

\makeatother

\newtheorem{theorem}{Theorem}
\newtheorem{lemma}{Lemma}

\newtheorem{corollary}{Corollary}

\theoremstyle{definition}

\newtheorem{conjecture}{Conjecture}

\begin{document}

\begin{center}
\uppercase{\bf A note on three consecutive powerful numbers}
\vskip 20pt
{\bf Tsz Ho Chan}\\
{\smallit Department of Mathematics, Kennesaw State University, Marietta, GA, U.S.A.}\\
{\tt tchan4@kennesaw.edu}\\ 
\end{center}
\vskip 20pt
\centerline{\smallit Received: , Revised: , Accepted: , Published: } 
\vskip 30pt

\centerline{\bf Abstract}
\noindent
This note concerns the non-existence of three consecutive powerful numbers. We use Pell equations, elliptic curves, and second-order recurrences to show that there are no such triplets with the middle term a perfect cube and each of the other two having only a single prime factor raised to an odd power.

\pagestyle{myheadings}
\markright{\smalltt INTEGERS: 25 (2025)\hfill}
\thispagestyle{empty}
\baselineskip=12.875pt
\vskip 30pt

\section{Introduction and Main Results}

A positive integer $n$ is {\it powerful} or {\it squareful} if $p^2 \mid n$ for all primes $p$ such that $p \mid n$. It is well-known that any powerful number can be factored uniquely as $n = a^2 b^3$ for some positive integer $a$ and squarefree number $b$. Here a number $n$ is {\it squarefree} if $p^2 \nmid n$ for all primes $p$. The following is an initial list of powerful numbers:
\[
1, \, 4, \, 8, \, 9, \, 16, \, 25, \, 27, \, 32, \, 36, \, 49, \, 64, \, 72, \, \ldots .
\]
Notice that $8$ and $9$ are consecutive powerful numbers. Indeed, there are infinitely many such pairs. For example, the solutions of the Pell equation
\[
x^2 - 2 y^2 = 1
\]
give consecutive powerful pairs $2y^2, x^2$ as $y$ is even by considering perfect squares modulo $4$. The interested readers may consult \cite{Go} and \cite{W} for more discussions on this topic. Next, one can ask if there are three consecutive powerful numbers.
\begin{conjecture}[Erd\H{o}s, Mollin, Walsh] \label{conj1}
No three consecutive powerful numbers exist.
\end{conjecture}
This appears to be very hard. Some relevant references are \cite{E}, \cite{MW} and \cite{Gr}. Conditional on the $abc$-conjecture, one can show that there are only finitely many triples of consecutive powerful numbers. In this note, we consider the special case of three consecutive powerful numbers of the form $x^3 - 1$, $x^3$, $x^3 + 1$ and prove the following.
\begin{theorem} \label{thm1}
There are no three consecutive powerful numbers of the form
\begin{equation} \label{main}
x^3 - 1 = p^3 y^2, \; \; x^3, \; \; x^3 + 1 =  q^3 z^2
\end{equation}
with primes $p, q$ and positive integers $x, y, z$.
\end{theorem}
\begin{corollary} \label{cor1}
The Diophantine equation $64 x^6 - 1 = p^3 q^3 y^2$ has no solution with integers $x, y$ and primes $p, q$.
\end{corollary}

The author hopes that this note will stimulate further studies of powerful numbers and Conjecture \ref{conj1}.

Throughout this paper, all variables are integers. The letters $p$ and $q$ stand for prime numbers. The symbol $a \mid b$ means that $a$ divides $b$, and $a \nmid b$ means that $a$ does not divide $b$. The symbol $p^k \, \| \, a$ means that $p^k \mid a$ but $p^{k+1} \nmid a$.

\section{Some Basic Observations}

\begin{lemma} \label{lem1}
The difference between any two perfect squares cannot be $2$.
\end{lemma}
\begin{proof} 
One simply observes that $(n+1)^2 - n^2 = 2n + 1 > 2$ when $n \ge 1$, and $(0 + 1)^2 - 0^2 = 1 \neq 2$.
\end{proof}

\begin{lemma} \label{lem2}
If $x \equiv 1 \pmod{3}$, then $3 \, \| \, x^2 + x + 1$.
\end{lemma}
\begin{proof}
Suppose $x \equiv 1 \pmod{3}$. Say $x = 3 x' + 1$ for some integer $x'$. Then
\[
x^2 + x + 1 = (3 x' + 1)^2 + (3 x' + 1) + 1 = 9 x'^2 + 9 x' + 3 = 3 (3 x'^2 + 3 x' + 1)
\]
which is divisible by $3$ but not $9$.
\end{proof}

\begin{lemma} \label{lem3}
If $x \equiv 2 \pmod{3}$, then $3 \, \| \, x^2 - x + 1$.
\end{lemma}
\begin{proof}
This follows from Lemma \ref{lem2} by the substitution $x \mapsto - x$.
\end{proof}

\begin{lemma} \label{lem4}
Suppose $a \neq 0$, $c$ and $e$ are some fixed integers. If $y^2 = a x^4 + c x^2 + e$ has an integer solution $(x, y)$ with $x \neq 0$, so does the elliptic curve $Y^2 = X^3 + c X^2 + a e X$.
\end{lemma}
\begin{proof}
One simply multiplies both sides of $y^2 = a x^4 + c x^2 + e$ by $a^2 x^2$ and gets $(a x y)^2 = (a x^2)^3 + c (a x^2)^2 + a e (a x^2)$. This yields a non-zero integer solution for the above elliptic curve with $X = a x^2 \neq 0$ and $Y = a x y$.
\end{proof}

\begin{lemma} \label{lem5}
The Pell equation $x^2 - 3 y^2 = 1$ has all positive integer solutions $(x_k, y_k)$ generated by $x_k + y_k \sqrt{3} = (2 + 1 \cdot \sqrt{3})^k$ for $k \in \mathbb{N}$. Moreover, the solutions satisfy the recursions: $x_1 = 2$, $x_2 = 7$, $x_{k} = 4 x_{k - 1} - x_{k - 2}$; $y_1 = 1$, $y_2 = 4$, $y_{k} = 4 y_{k - 1} - y_{k - 2}$ for $k > 2$.
\end{lemma}
\begin{proof}
This is a standard result in the theory of Pell equation that all integer solutions are generated by some fundamental (minimal) solution. See \cite[Theorem 5.3]{C1} for example. The recursions easily follow from the observation
\[
(2 + \sqrt{3})^{k} - 4 (2 + \sqrt{3})^{k-1} + (2 + \sqrt{3})^{k-2} = (2 + \sqrt{3})^{k-2} [ (2 + \sqrt{3})^2 - 4 (2 + \sqrt{3}) + 1 ] = 0
\]
as $2 + \sqrt{3}$ is a root to the quadratic equation $x^2 - 4 x + 1 = 0$.
\end{proof}

\begin{lemma} \label{lem6}
The generalized Pell equation $x^2 - 3y^2 = -2$ has all positive integer solutions $(x_k, y_k)$ generated by $x_k + y_k \sqrt{3} = (1 + 1 \cdot \sqrt{3}) (2 + 1 \cdot \sqrt{3})^k$ for $k \in \mathbb{N}$. Moreover, the solutions satisfy the recursions: $x_1 = 1$, $x_2 = 5$, $x_{k} = 4 x_{k - 1} - x_{k - 2}$; $y_1 = 1$, $y_2 = 3$, $y_{k} = 4 y_{k - 1} - y_{k - 2}$ for $k > 2$.
\end{lemma}
\begin{proof}
One can easily see that $x = 1 = y$ is the smallest positive integer solution (i.e., $x_1 + y_1 \sqrt{3} = 1 + 1 \cdot \sqrt{3}$). Then one can generate all the integer solutions by combining this with the solutions in Lemma \ref{lem5}. See \cite[Theorem 3.3]{C2} for example. The recursions follow from a similar observation as in Lemma \ref{lem5}.
\end{proof}

\section{Proof of Theorem \ref{thm1}}

We are now ready to prove Theorem \ref{thm1}.

\begin{proof}[Proof of Theorem \ref{thm1}]
First, note that any three consecutive powerful numbers must be of the form $4n - 1$, $4n$, $4n + 1$ as $2 \, \| \, 4n + 2$. Recall that the middle number is a cube. So, $4n = x^3$ and $x$ is even. Suppose that, contrary of Theorem \ref{thm1}, we have
\begin{equation} \label{main1}
(x - 1) (x^2 + x + 1) = p^3 y^2 \; \text{ and } \; (x+1) (x^2 - x + 1) = q^3 z^2.
\end{equation}
Note that $x > 1$ as $p, y > 0$. By the fact that $\gcd(a, b) = \gcd(a, b - a)$, we have $\gcd(x - 1, x^2 + x + 1) = \gcd(x - 1, (x - 1)(x + 2) + 3) = \gcd(x - 1, 3) = 1$ or $3$.
Hence,
\begin{equation} \label{gcd1}
\gcd(x - 1, x^2 + x + 1) = \left\{ \begin{array}{ll} 3, & \text{ if } x \equiv 1 \pmod{3}, \\
1, & \text{ otherwise}. \end{array} \right.
\end{equation}

Suppose $x \not\equiv 1 \pmod{3}$. Then $\gcd(x - 1, x^2 + x + 1) = 1$ by Equation \eqref{gcd1}. It follows that if $p \mid x - 1$, then $p^3 \mid x - 1$ and $x^2 + x + 1$ is a perfect square by Equation \eqref{main1}. This is impossible as $x^2 < x^2 + x + 1 < (x + 1)^2$. Hence, $p^3 \mid x^2 + x + 1$. In summary, we have
\begin{equation} \label{imply1}
x \not\equiv 1 \pmod{3} \; \text{ implies } \; p^3 \mid x^2 + x + 1 \text{ and } x - 1 \text{ is a perfect square}.
\end{equation}
Applying the substitution $x \mapsto -x$ to the above argument, we also have
\begin{equation} \label{gcd2}
\gcd(x + 1, x^2 - x + 1) = \left\{ \begin{array}{ll} 3, & \text{ if } x \equiv 2 \pmod{3}, \\
1, & \text{ otherwise}, \end{array} \right.
\end{equation}
and
\begin{equation} \label{imply2}
x \not\equiv 2 \pmod{3} \; \text{ implies } \; q^3 \mid x^2 - x + 1 \text{ and } x + 1 \text{ is a perfect square}.
\end{equation}
\noindent{\tt Case 1:} $x \equiv 0 \pmod 3$. From Equations \eqref{gcd1} and \eqref{gcd2}, we have
\[
\gcd( x - 1, x^2+x+1 ) = 1 = \gcd(x+1, x^2-x+1).
\]
By Equations \eqref{imply1} and \eqref{imply2}, both $x - 1$ and $x + 1$ are perfect squares which contradicts Lemma \ref{lem1}. 

\noindent{\tt Case 2:} $x \equiv 1\pmod 3$. From Equations \eqref{gcd1} and \eqref{gcd2}, we have
\[
\gcd(x - 1, x^2+x+1) = 3 \; \text{ and } \; \gcd(x+1, x^2-x+1) = 1.
\]
Hence, $q^3 \mid x^2 - x + 1$ and $x + 1 = u^2$ for some integer $u > 0$ by Equation \eqref{imply2}.

Suppose $p = 3$. We have $(x - 1)(x^2 + x + 1) = 3^3 y^2$ and $3 \, \| \, x^2 + x + 1$ by Lemma \ref{lem2}. Thus, $9 \mid x - 1$ and $x - 1 = 9 x'^2 = (3 x')^2$ for some integer $x'$. This contradicts Lemma \ref{lem1}.

Suppose $p \neq 3$. If $p \mid x^2 + x + 1$, then $p^3 \mid x^2 + x + 1$ by Equation \eqref{gcd1}. This forces $x - 1 = 3 v^2$ or $x - 1 = (3 v)^{2}$ for some integer $v$. The latter contradicts Lemma \ref{lem1}. The former yields $u^2 - 3v^2 = 2$ which is also impossible as $u^2 \equiv 0$ or $1 \pmod{3}$.

Therefore, we must have $p \mid x - 1$ and, hence, $p^3 \mid x - 1$ by Equation \eqref{gcd1}. From Equation \eqref{main1} and Lemma \ref{lem3}, we have $x^2 + x + 1 = 3 v^2$ and $x - 1 = 3 p^3 w^2$ for some integers $v, w > 0$ with $\gcd(v, 3) = 1 = \gcd(v, w)$. Substituting $x = u^2 - 1$ into $x^2 + x + 1$, we obtain
\[
3 v^2 = u^4 - u^2 + 1 \; \; \text{ or } \; \; y^2 = 3 u^4 - 3 u^2 + 3
\]
where $y = 3 v$. By Lemma \ref{lem4}, the elliptic curve $Y^2 = X^3 - 3 X^2 + 9 X$ has a non-zero integer solution. This contradicts $(X, Y) = (0, 0)$ being the only integer solution by the SageMath command
\[
\text{E = EllipticCurve([0,-3,0,9,0]); E.integral\_points()},
\]
for example.

\noindent{\tt Case 3:} $x \equiv 2 \pmod 3$. From Equations \eqref{gcd1} and \eqref{gcd2}, we have
\[
\gcd(x - 1, x^2+x+1) = 1, \; \text{ and } \; \gcd( x + 1, x^2 - x + 1) =  3.
\]
Hence, $p^3 \mid x^2 + x + 1$ and $x - 1 = u^2$ for some integer $u > 0$ with $3 \nmid u$ by Equation \eqref{imply1}.

Suppose $q = 3$. We have $(x + 1)(x^2 - x + 1) = 3^3 y^2$ and $3 \, \| \, x^2 - x + 1$ by Lemma \ref{lem3}. Thus, $9 \mid x + 1$ and $x + 1 = 9 x'^2 = (3 x')^2$ for some integer $x'$. This contradicts Lemma \ref{lem1}.

Suppose $q \neq 3$. We have either $q^3 \mid x+1$ or $q^3 \mid x^2-x+1$ by Equation \eqref{gcd2}. If the former is true, then $x + 1 = 3 q^3 v^2$ and $x^2 - x + 1 = 3 w^2$ for some integers $v, w > 0$ with $\gcd(w, 3) = 1 = \gcd(v, w)$. Substituting $x = u^2 + 1$ into $x^2 - x + 1$, we obtain
\[
3 w^2 = u^4 + u^2 + 1 \; \; \text{ or } \; \; y^2 = 3 u^4 + 3 u^2 + 3
\]
with $y = 3w$. By Lemma \ref{lem4}, the elliptic curve $Y^2 = X^3 + 3 X^2 + 9 X$ has a non-zero integer solution. The SageMath command
\[
\text{E = EllipticCurve([0,3,0,9,0]); E.integral\_points()}
\]
yields $(X, Y) = (0, 0)$ and $(3, \pm 9)$ as the only such solutions. It follows from the proof in Lemma \ref{lem4} that $X = 3 u^2 = 3$ and $u = 1$. Hence, $x = u^2 + 1 = 2$ but $x^3 - 1 = 7$ is not powerful.

Therefore, we must have $q^3 \mid x^2 - x + 1$. By Equation \eqref{gcd2}, we have $x+1 = 3 v^2$ and $x^2 - x + 1 = 3 q^3 w^2$ for some integers $v$ and $w$. Combining these with $x - 1 = u^2$, we get the generalized Pell equation $u^2 - 3 v^2 = -2$. By Lemma \ref{lem6}, $u = u_k$ satisfies
\begin{equation} \label{u-recur}
u_1 = 1, \; \; u_2 = 5, \; \text{ and  } \; u_{k} = 4 u_{k-1} - u_{k-2} \; \text{ for } \; k > 2.
\end{equation}
By substituting $x = u^2 + 1$ into $x^2 - x + 1$, we obtain
\[
3 q^3 w^2 = (u^2+1)^2 - (u^2+1) + 1 = u^4 + u^2 + 1 = (u^2 + u + 1)(u^2 - u + 1).
\]
Since $x$ is even and $x - 1 = u^2$, $u$ is odd and
\begin{equation} \label{gcd3}
\gcd (u^2 + u + 1, u^2 - u + 1) = \gcd(u^2 + u + 1, 2u) = \gcd(u^2 + u + 1, u) = 1.
\end{equation}
\noindent{\tt Subcase 1:} $u \equiv 1 \pmod{3}$. Then $3 \mid u^2 + u + 1$. Suppose $q \mid u^2 + u + 1$. We have $u^2 + u + 1 = 3 q^3 w_1^2$ and $u^2 - u + 1 = w_2^2$ by Equation \eqref{gcd3}. This is impossible as $(u - 1)^2 < u^2 - u + 1 < u^2$ unless $u = 1$. However, $u = 1$ yields $x = 2$ and $x^3 - 1 = 7$ which is not powerful. Therefore, we must have $q \mid u^2 - u + 1$. So, $u^2 + u + 1 = 3 w_1^2$ and $u^2 - u + 1 = q^3 w_2^2$ by Equation \eqref{gcd3}. After some algebra, one arrives at
\[
(2w_1)^2 - 3 \Bigl( \frac{2u+1}{3} \Bigr)^2 =  1.
\]
Then Lemma \ref{lem5} gives $2w_1 + \frac{2u+1}{3} \sqrt{3} = g_l + h_l \sqrt{3} = (2 + \sqrt{3})^l$ where
\begin{equation} \label{h-recur}
h_1 = 1, \; \; h_2 = 4, \; \text{ and } \; h_{l} = 4 h_{l-1} - h_{l-2} \; \text{ for } \; l > 2.
\end{equation}
Thus, $u_k = u = \frac{3 h_{l} - 1}{2}$ for some indices $k, l \ge 1$.

From Equations \eqref{u-recur} and \eqref{h-recur}, one can show by induction that $u_k$ and $h_l$ are positive increasing sequences (for example, $u_2 > u_1 > 0$ and the induction hypothesis $u_{k-1} > u_{k-2} > 0$ implies $u_{k} = 4 u_{k-1} - u_{k-2} > u_{k-1} + 3 (u_{k-1} - u_{k-2}) > u_{k-1} > 0$). Hence,
\begin{equation} \label{increase}
u_k = 4 u_{k-1} - u_{k-2} > 4 u_{k-1} - u_{k-1} = 3 u_{k-1}, \text{ and, similarly,  } h_l > 3 h_{l-1}.
\end{equation}
Moreover, one can form the new sequence $v_k := u_k - h_k$ which satisfies
\[
v_1 = 0, \; \; v_2 = 1, \; \text{ and } \; v_{k} = 4 v_{k-1} - v_{k-2} \; \text{ for } \; k > 2.
\]
As $v_3 = 4$, one can see that $v_k = h_{k-1} > 0$ for $k \ge 2$. Hence, $u_k = h_k + h_{k-1} > h_k$ for all $k \ge 2$. By Equation \eqref{increase} and the inequality $\frac{4n}{3} < \frac{3n-1}{2}$ when $n \ge 4$, we have
\[
u_k = h_k + h_{k-1} < \frac{4 h_k}{3} < \frac{3 h_{k} - 1}{2} < \frac{3 u_k - 1}{2} < 3 u_k <  u_{k+1}
\]
for all $k \ge 2$. Therefore, $u_k = \frac{3 h_{l} - 1}{2}$ is possible only when $k = l = 1$. This gives $x = u_1^2 + 1 = 2$ but $x^3 - 1 = 7$ is not powerful.

\noindent{\tt Subcase 2:} $u \equiv -1 \pmod{3}$. This is very similar to subcase 1 with $3 \mid u^2 - u + 1$ and $(2 w_1)^2 - 3 (\frac{2u - 1}{3})^2 = 1$ instead. It also yields a contradiction.
\end{proof}
\section{Proof of Corollary \ref{cor1}}

Using Theorem \ref{thm1}, we can now prove Corollary \ref{cor1}.
\begin{proof}[Proof of Corollary \ref{cor1}]
Suppose the equation $64 x^6 - 1 = ((2x)^3 - 1) ((2x)^3 + 1) = p^3 q^3 y^2$ has a solution with some integers $x, y$, and primes $p, q$. By the fact that $\gcd(a, b) = \gcd(a, b - a)$, we have
\begin{equation} \label{gcdlast}
\gcd((2x)^3 - 1, (2x)^3 + 1) = \gcd((2x)^3 - 1, 2) = 1.
\end{equation}

Suppose $p q \mid (2x)^3 - 1$. By Equation \eqref{gcdlast}, we must have $(2x)^3 - 1 = p^3 q^3 y_1^2$ and $(2x)^3 + 1 = y_2^2$ for some integers $y_1$ and $y_2$. However, the elliptic curve $Y^2 = X^3 + 1$ has $(X, Y) = (-1, 0)$, $(0, \pm 1)$ and $(2, \pm 3)$ as its only integer solutions by SageMath for example. Thus, $x = 0$ or $1$. However, neither $0^6 - 1 = -1$ nor $2^6 - 1 = 63$ are of the form $p^3 q^3 y_1^2$.

Suppose $p q \mid (2x)^3 + 1$. By Equation \eqref{gcdlast}, we must have $(2x)^3 + 1 = p^3 q^3 y_1^2$ and $(2x)^3 - 1 = y_2^2$ for some integers $y_1$ and $y_2$. This contradicts $(X,Y) = (1,0)$ being the only solution to the elliptic curve $Y^2 = X^3 - 1$ (see \cite[Theorem 3.2]{C3} for example).

Therefore, $p$ divides exactly one of $(2x)^3 - 1$ or $(2x)^3 + 1$, and $q$ divides the other one. Without loss of generality, say $(2x)^3 - 1 = p^3 y_1^2$ and $(2x)^3 + 1 = q^3 y_2^2$. This contradicts Theorem \ref{thm1}. Consequently, $64 x^6 - 1 = p^3 q^3 y^2$ cannot have any integer solution.
\end{proof}
\vskip 20pt\noindent {\bf Acknowledgement.} The author would like to thank the anonymous referee and the managing editor Bruce Landman for some helpful suggestions.


\end{document}